\documentclass{amsart}
\usepackage{fdsymbol} 
\usepackage{tikz-cd} 
\usepackage{stmaryrd} 

    \usepackage{etoolbox}
    \patchcmd{\section}{\scshape}{\large\bfseries}{}{}
    \makeatletter
    \renewcommand{\@secnumfont}{\bfseries}
    \makeatother
    
\usepackage{biblatex}
\usepackage{booktabs}
\bibliography{references}

\numberwithin{equation}{section}

\newtheorem{theorem}{Theorem}[section]
\newtheorem*{theorem*}{Theorem}
\newtheorem{corollary}[theorem]{Corollary}
\newtheorem{lemma}[theorem]{Lemma}
\newtheorem{proposition}[theorem]{Proposition}

\theoremstyle{definition}
\newtheorem{definition}[theorem]{Definition}
\newtheorem{remark}[theorem]{Remark}

\def\ZZ{\mathbb{Z}}
\def\QQ{\mathbb{Q}}
\def\epi{\twoheadrightarrow}

\def\Ext{\mathsf{Ext}}
\def\Ker{\mathsf{Ker}}
\def\gg{\mathfrak{g}}
\def\ff{\mathfrak{f}}
\def\ll{\mathfrak{l}}
\def\rr{\mathfrak{r}}
\def\uu{\mathfrak{u}}

\DeclareMathOperator{\rk}{\mathrm{rk}}
\DeclareMathOperator{\spn}{\mathrm{span}}

\usepackage{stmaryrd}
\def\llb{\llbracket}
\def\rrb{\rrbracket}
\def\llp{\llparenthesis}
\def\rrp{\rrparenthesis}
\def\Sym{\mathsf{Sym}}

\usepackage{xcolor}

\usepackage{hyperref}
\hypersetup{
    colorlinks=false,
    linkcolor=blue,
    filecolor=magenta,      
    urlcolor=cyan,
    citecolor=blue,
}

\title{Homology of the pronilpotent completion and cotorsion groups}

\author{Mikhail Basok} 
\address{
Laboratory of Modern Algebra and Applications,  St. Petersburg State University, 14th Line, 29b,
Saint Petersburg, 199178 Russia}
\email{m.k.basok@gmail.com}

\author{Sergei O. Ivanov} 
\address{
Laboratory of Modern Algebra and Applications,  St. Petersburg State University, 14th Line, 29b,
Saint Petersburg, 199178 Russia}
\email{ivanov.s.o.1986@gmail.com}

\author{Roman Mikhailov} 
\address{
Laboratory of Modern Algebra and Applications,  St. Petersburg State University, 14th Line, 29b,
Saint Petersburg, 199178 Russia, and St. Petersburg Department of Steklov Mathematical
Institute}
\email{rmikhailov@mail.ru}

\thanks{The work is supported by the grant of the Government of the Russian Federation for the
state support of scientific research carried out under the supervision of leading scientists, agreement 14.W03.31.0030
dated 15.02.2018.1}

\begin{document}

\maketitle

\begin{abstract}
For a non-cyclic free group $F$, the second homology of its pronilpotent completion $H_2(\widehat F)$ is not a cotorsion group.    
\end{abstract}

\section{Introduction}

\subsection{Main result} Given a group $G$, denote by $\{\gamma_i(G)\}_{i\geq 1}$ its lower central series $\gamma_1(G):=G,$ $\gamma_{i+1}(G):=[\gamma_i(G),G],\ i\geq 1$. The inverse limit $\widehat G:={\varprojlim}\ G/\gamma_i(G)$ is called the {\it pronilpotent completion of $G$}. In this paper we study free pronilpotent completions $\widehat F$, that is the pronilpotent completions of free groups $F$. Our main result is the following:

\vspace{.5cm}\noindent{\bf Theorem A.} For a non-cyclic free group $F$, the second integral homology $H_2(\widehat F)$ is not a cotorsion group. 

\vspace{.5cm}
\noindent [For integral homology we omit the coefficients, i.e. $H_*(-)=H_*(-,\mathbb Z).$] Recall that an abelian group $A$ is called \emph{cotorsion}, if ${\sf Ext}(\mathbb Q,A)=0.$ The class of cotorsion groups coincides with the class of values of $\varprojlim^1$-functor for inverse sequences of abelian groups \cite{warfield1979values}. As a corollary, we see that $H_2(\widehat{F})$ can not be presented as $\varprojlim^1$ of an inverse sequence of abelian groups $A_1 \leftarrow A_2 \leftarrow \dots $
\begin{equation}
H_2(\widehat{F}) \not\cong {\varprojlim}^1 A_i.
\end{equation}
In particular, for the case of pronilpotent completion there is no chance to get a Milnor-type exact sequence 
$$
{\varprojlim}^1 H_{*+1}\to H_*(\varprojlim)\to \varprojlim H_*
$$

We also prove the following statement which shows that the property to be a cotorsion group appears naturally in the context of homology of the pronilpotent completions. 

\vspace{.5cm}\noindent{\bf Theorem B.}
Let $G$ be a finitely generated group and $F\epi G$ be a presentation, where $F$ is a free group of finite rank. Then the cokernel of the map
\begin{equation}
H_2(\widehat{F})\oplus H_2(G) \longrightarrow H_2(\widehat{G})
\end{equation}
is cotorsion. 
\vspace{.5cm}

Next we give some motivation for the studying the free nilpotent completion and in particular its homological properties. 

\subsection{Group theory} 
In 1960-s G. Baumslag initiated a study of parafree groups. A group is called parafree if it is residually nilpotent and has the same lower central quotients as a free group \cite{baumslag1967some}, \cite{baumslag1967groups}, \cite{baumslag1969groups}, \cite{baumslag1968more}. What are the properties parafree groups inherit from free groups? More general, what group-theoretical properties can one extract from the structure of the lower central quotients? These problems were studied by G. Baumslag during decades. For a free group $F$, all parafree groups with the same lower central quotients as $F$, are contained in the free pronilpotent completion $\widehat F$. If $F$ is finitely generated, $\widehat F$ is itself parafree. 

The free pronilpotent completion $\widehat F$ is an extremely complicated group. A.K. Bousfield proved in \cite{bousfield1977homological} that its integral second homology $H_2(\widehat F)$ is uncountable.  What is the cohomological dimension of $\widehat F$? Is it true that $cd(\widehat F)=2$? Is it true that $H_3(\widehat F)=0$? These are open problems. These problems can be viewed as a part of Baumslag's program of the study of properties of parafree groups. The problem about the structure of $H_2(\widehat F)$ also is intriguing. As a part of the proof of uncountablility of $H_2(\widehat F)$ in \cite{bousfield1977homological}, it is shown that there exists an epimorphism from $H_2(\widehat F)$ to the exterior square of 2-adic integers. The exterior square of 2-adic integers is an uncountable divisible torsion-free group. In \cite{ivanov2018discrete} it is shown that $H_2(\widehat F)$ is not divisible. It seems that nothing more is published about the structure of $H_2(\widehat F)$. 

The homology of a group are related to its lower central quotients by the classical Stallings theorem \cite{stallings1965homology}: a homomorphism of two groups which induces an isomorphism of $H_1$ and an epimorphism of $H_2$ induces isomorphism of the lower central quotients. In particular, for a group with a free abelian $H_1$ and zero $H_2$, the lower central quotients are free nilpotent. Here is a simple remark which gives a generalization of the above statement. 
\begin{remark}\label{remark1}
Let $G$ be a group with a free abelian $H_1(G)$ and ${\sf Hom}(H_2(G),\mathbb Z)=0$. Then there exists a free group $F$ and a homomorphism $F\to G$ which induces the isomorphisms of the lower central quotients. 
\end{remark}
[We give the proofs of the remarks from this section at the end of the paper.]

In particular, a residually nilpotent group $G$ with a free abelian $H_1(G)$ and ${\sf Hom}(H_2(G),\mathbb Z)=0$ is parafree. Is there a chance to convert this statement? In particular, can one show that ${\sf Hom}(H_2(\widehat F),\mathbb Z)=0$? We leave these questions in the form of conjectures. 

\vspace{3mm}\noindent {\bf Conjecture 1.} For a non-cyclic free group $F$ the following holds 
\begin{equation}
{\sf Hom}(H_2(\widehat F),\ZZ)=0.    
\end{equation}

\noindent {\bf Conjecture 2.} Let $G$ be a residually nilpotent group. Then $G$ is parafree if and only if $H_1(G)$ is free abelian and ${\sf Hom}(H_2(G),\ZZ)=0.$ 
\vspace{3mm}

Conjecture 1 would imply that $H_2(\widehat{F})$ is not free abelian, and hence, the cohomological dimension of $\widehat F$ would be greater than two. The class of cotorsion groups is a subclass of groups with the property ${\sf Hom}(-, \mathbb Z)=0$. Our Theorem A shows that $H_2(\widehat F)$ is not from this subclass.

\subsection{Low-dimensional topology} In order to find transifinte $\bar\mu$-invariants of links, K. Orr introduced the space $K_\infty$ \cite{orr1987new}. The space $K_\infty$ is the mapping cone of the natural map $K(F,1)\to K(\widehat F,1)$. 
The homotopy groups $\pi_i(K_\infty),\ i\geq 3,$ are repositories of potential invariants of links (see \cite{cochran1987link}). In the case of classical links, the invariants lie in $\pi_3(K_\infty)$. 
The group $\pi_3(K_\infty)$ is infinite \cite{dror2018third}, however its structure is far from being clear. 
In particular, the Hurewicz homomorphism $\pi_3(K_\infty)\to H_3(\widehat F)$ is an epimorphism.  
That is, the existence of non-zero elements in the (higher) homology of free pronilpotent completions may have application in low-dimensional topology.   

\subsection{Bousfield-Kan theory} The free pronilpotent completion $\widehat F$ appears naturally in the theory of Bousfield-Kan \cite{bousfield1972homotopy},  constructions of localizations and completions for spaces. 
 
Recently Barnea and Shelah proved that, for any
 sequence of epimorphisms $G_{i+1}\to G_i,\ i\geq 1$, the kernel and cokernel of the natural map
$$
H_1(\varprojlim\ G_i)\to \varprojlim H_1(G_i)
$$ 
are cotorsion groups \cite{barnea2018abelianization}. The next remark is a consequence of this statement. 

\begin{remark}\label{remark2} For a connected space $X$, the cokernel of the natural map $$H_1(X)\to H_1({\mathbb Z}_\infty X)
$$  
is a cotorsion group. Here $\mathbb Z_\infty X$ is the integral Bousfield-Kan completion of $X$.  
\end{remark}

For a free group $F$, $\mathbb Z_\infty K(F,1)=K(\widehat F,1)$ (see \cite{bousfield1972homotopy}, Section IV). Theorem A implies that, the cokernel of the natural map 
$$
H_2(K(F,1))\to H_2(\mathbb Z_\infty K(F,1))=H_2(K(\widehat F,1))
$$
is not a cotorsion group. That is, the above remark can not be extended to the second homology. 

\subsection{Lie algebras} For a Lie algebra $\gg$ over $\ZZ$, consider its lower central series $\{\gamma_i(\gg)\}_{i\geq 1}$ and its pronilpotent completion $\widehat \gg:={\varprojlim}\ \gg/\gamma_i(\gg)$. Most of the discussed above problems for groups can be acked for Lie algebras as well. In particular, for the moment we do not understand the structure of homology of a free pronilpotent completion in the case of Lie algebras. We are able to prove the following:

\vspace{.5cm}\noindent{\bf Lie analog of Theorem A.} For a non-cyclic free Lie algebra $\ff$ over integers, $H_2(\widehat \ff)$ is not a torsion group. 
\vspace{.5cm}

In \cite{ivanov2020homological} it is shown that $H_2(\widehat \ff)$ is uncountable. The method used in \cite{ivanov2020homological} is similar to one from \cite{ivanov2019finite}, the authors present explicit cycles in the Chevalley-Eilenberg complex 
$$
\widehat \ff\wedge \widehat \ff\wedge \widehat \ff\to \widehat \ff\wedge \widehat \ff\to \widehat \ff
$$
and show that these cycles are not boundaries. In particular, divisible elements in $H_2(\widehat \ff)$ are constructed in \cite{ivanov2020homological} in this way. Let $\ff$ be a free Lie algebra on generators $a_1,\dots, a_n$. All the elements in $H_2(\widehat \ff)$ constructed in \cite{ivanov2020homological} are of the form 
\begin{equation}\label{formofelements}
\alpha_1\wedge a_1+\dots+\alpha_n\wedge a_n, 
\end{equation}
for some (infinite) $\alpha_i\in \widehat \ff$. 

\begin{remark}\label{remark3}
In the above notation, the subgroup in $H_2(\widehat \ff)$ generated by elements of the form (\ref{formofelements}) is cotorsion. 
\end{remark}

That is, the group $H_2(\widehat \ff)$ contains a huge uncountable cotorsion subgroup.

\subsection{Structure of the paper} The paper is organized as follows. In Section 2 we recall the properties of cotorsion groups and give a proof of Theorem B. It turns out that Theorem B follows from the generalized Hopf formula and the result of Barnea and Shelah. Section 3 is the most complicated section of the paper, it contains technical results about power series. The main result of Section 3 is Theorem 3.1. In Section 4 we consider the integral lamplighter group $L=\ZZ \wr \ZZ$. The main result of Section 4 is Theorem 4.4 which states that the cokernel of the natural map $H_2(L)\to H_2(\widehat L)$ is not a cotorsion group. Theorem 3.1 appears there as a key technical ingredient. Theorem A now follows immediately from Theorem B and Theorem 4.4. Observe that, the cokernel of the map $H_2(\widehat F)\oplus H_2(L)\to H_2(\widehat L)$ which is a cotorsion group by Theorem B, is huge, for instance it can be shown that it maps onto the quotient $\mathbb Z\llb x\rrb/\mathbb Z [x].$ In Section 6 we overview the arguments needed to prove the analogs of Theorems A and B for the case of Lie algebras. At the end of the paper, in Section 7, we give proofs of three Remarks written in Introduction.

\section{Cotorsion groups}

\subsection{Background about cotorsion groups}
\label{subsec:cotorsion_general}
Here we remind the definition and some facts about cotorsion abelian groups that can be found in the book of Fuchs \cite[Ch. IX,\S 54]{fuchs1970infinite}. We will also menstion a result of Warfield and Huber \cite{warfield1979values}. 

An abelian group $C$ is called \emph{cotorsion}, if $\Ext(\QQ,C)=0.$ The following properties are equivalent. 

\begin{enumerate}
\item $C$ is cotorsion;
\item $\Ext(A,C)=0$ for any torsion free group $A;$
\item $C\cong {\varprojlim}^1 A_i$ for some inverse sequence of abelian groups $A_1 \leftarrow A_2 \leftarrow \dots$ (see \cite{warfield1979values}). 
\end{enumerate}
Examples of cotorsion groups: divisible groups; finite groups; bounded groups; moreover, a reduced torsion group is cotorsion iff it is bounded;  for any abelian groups $A,B$ the group $\Ext(A,B)$ is cotorsion; a quotient of a cotorsion group is cotorsion; product of a family of groups is cotorsion iff each of them is cotorsion;  inverse limit of reduced cotorsion groups is a reduced cotorsion group; in particular, the group of $p$-adic integers $\ZZ_p=\varprojlim \ZZ/p^i$ is a reduced cotorsion group; for any sequence of abelian groups $A_1,A_2,\dots$ the quotient $(\prod A_i)/(\bigoplus A_i)$ is cotorsion. 

Non-examples of cotorsion groups: $\ZZ$ is not cotorsion; moreover, any group $A$ with a non-trivial homomorphism  $A\to \ZZ$ is not cotorsion; the infinite direct sum indexed by all primes  $\bigoplus_p \ZZ/p$ is not cotorsion.  

The notion of cotorsion group is closely related to the notion of algebraically compact group. A group $C$ is called \emph{algebraically compact}, if ${\sf Pext}(A,C)=0$ for any $A.$ Since for a torsion free abelian group $A$ we have ${\sf Ext}(A,C)={\sf Pext}(A,C),$ we see that any algebraically compact group is cotorsion. On the other hand any cotorsion group is a quotient of an algebraically compact group. Moreover, a torsion free abelian group is cotorsion if and only if it is algebraically compact.

\subsection{Cotorsion quotients of inverse limits} A theorem of Hulanicki \cite{hulanicki1962structure},  \cite[Cor. 42.2]{fuchs1970infinite} says that for a sequence of abelian group $A_1,A_2,\dots $ the  quotient of the infinite product by the infinite direct sum $(\prod A_i) / (\bigoplus A_i)$ is algebraically compact, and hence, it is cotorsion. Barnea and Shellah in \cite[Th 2.0.4]{barnea2018abelianization} proved a version of this result about inverse limits of non-abelian groups. For an inverse sequence of groups and epimorphisms $G_1 \twoheadleftarrow G_2 \twoheadleftarrow \dots  $ a subgroup of the inverse limit $U\subseteq \varprojlim G_i$ is dense in the inverse limit topology if and only if the maps $U\to G_i$ are surjective. Then the theorem of Barnea and Shelah can be reformulated as follows. 

\begin{theorem}[{\cite[Th 2.0.4]{barnea2018abelianization}}] \label{th:barnea_shelah}
Let $G_1 \twoheadleftarrow G_2 \twoheadleftarrow \dots  $ be an inverse sequence of groups and epimorphisms, $G_\infty=\varprojlim G_i$ and $U$ be a dense subgroup of $G_\infty$ containing $[G_\infty,G_\infty].$ Then $G_\infty / U$ is a cotorsion abelian group. 
\end{theorem}

\subsection{Generalised Hopf's formula} The classical Hopf's formula says that for a group presented as a quotient of a free group $G=F/R$ its second integral homology can be computed as follows $H_2(G)\cong (R\cap [F,F])/[R,F].$ The following generalization of this formula is seems to be well known (c.f. \cite[ Lemma 2.3]{ivanov2018discrete}) but we could not find a reference. 

\begin{proposition}[Hopf's formula]\label{prop:hopf}
Let $U$ be a normal subgroup of a group $G,$ then there is an exact sequence
\begin{equation}
H_2(G) \to H_2(G/U) \to
\frac{U\cap [G,G]}{[U,G]}
 \to 0.
\end{equation}
\end{proposition}
\begin{proof}
Consider the short exact sequence $1\to U\to G\to G/U\to 1$ and the corresponding five-term exact sequence of the Lyndon–Hochschild–Serre spectral sequence
\begin{equation}
H_2(G) \to H_2(G/U) \to H_0(G/U,H_1(U)) \to H_1(G) \to H_1(G/U) \to 0.  
\end{equation}
Simple computations show that there are isomorphisms
\begin{equation}
H_0(G/U,H_1(U))\cong U/[U,G], \hspace{1cm} H_1(G)\cong G/[G,G]  
\end{equation}
and the map $H_0(G/U,H_1(U))\to H_1(G)$ corresponds to the  map $U/[U,G]\to G/[G,G]$ induced by embedding maps. Therefore, the kernel of this map is isomorphic to $(U\cap [G,G])/[U,G].$
\end{proof}

\subsection{Proof of Theorem B}
\label{subsection:proof1}

Set $R=\Ker(F\epi G),$ $\overline{R}:=\Ker(\widehat{F} \to \widehat{G})$ and $F^n=\gamma_n(F),$  $ G^n=\gamma_n(G).$ If we apply the limit to the short exact sequence 
\begin{equation}
1 \longrightarrow R/(R\cap F^n)  \longrightarrow F/F^n \longrightarrow G/G^n \longrightarrow 1 
\end{equation}
and use that the functor of limit is exact for towers of epimorphisms,
we obtain an isomorphism 
\begin{equation}
\overline{R}\cong \varprojlim R/(R\cap F^n)
\end{equation}
and the fact that  the map $\widehat{F}\to 
\widehat{G}$ is an epimorphism. Therefore by the generalised Hopf's formula (Proposition \ref{prop:hopf}) there is an exact sequence 
\begin{equation}
H_2(\widehat{F}) \to H_2(\widehat{G}) \to \frac{ \overline{R}\cap \widehat{F}^2 }{[\overline{R},\widehat{F}]} \to 0.
\end{equation}
We also have an isomorphism $H_2(G)=(R\cap F^2)/[R,F].$ Therefore it is sufficient to prove that the cokernel of the map
\begin{equation}
\frac{ R\cap F^2 }{[R,F]}    \longrightarrow  \frac{ \overline{R}\cap \widehat{F}^2 }{[\overline{R},\widehat{F}]}
\end{equation}
is cotorsion.

Prove it. For a finitely generated free group $F$ we have an isomorphism $F_{ab}\cong \widehat{F}_{ab}$ (see \cite[Th. 13.3 (iv)]{bousfield1977homological}, \cite[Th.2.1]{baumslag1977inverse}).
So, if we take the limit of the short exact sequence
\begin{equation}
1 \to  (R\cap F^2)/(R\cap F^n)  \to  R/(R\cap F^n)    \to F_{ab} 
\end{equation}
and use that the functor of limit is left exact, we obtain an exact sequence
\begin{equation}
1 \to  \varprojlim (R\cap F^2)/(R\cap F^n)  \to  \overline{R}    \to \widehat{F}_{ab}.
\end{equation}
It follows that 
\begin{equation}
\overline{R} \cap \widehat{F}^2 \cong   \varprojlim \  (R\cap F^2)/(R\cap F^n). 
\end{equation}
Note that the image of $R\cap F^2$ is dense in $\overline{R}\cap \widehat{F}^2$ with respect to the inverse limit topolgy. Denote this image by $I.$ Then the subgroup $I[\overline{R},\widehat{F}]$ of $\overline{R}\cap \widehat{F}^2$ satisfies assumptions of Theorem \ref{th:barnea_shelah}. Hence, the group $(\overline{R}\cap \widehat{F}^2)/(I[\overline{R},\widehat{F}])$ is cotorsion. The assertion follows.

\section{A technical result about power series}

\subsection{Formulation} For a commutative ring $R$ we denote by  $R\llb x\rrb$ be the ring of formal power series in the variable $x$ over $R$. Its elements will be called just power series. We also denote by $R^\times$ the group of invertible elements of $R.$   It is well known that a power series $f=\sum a_i x^i\in R\llb x \rrb$ is invertible if and only if its constant term $a_0$ is invertible in $R:$
\begin{equation}
R\llb x \rrb^\times = \left\{ \sum a_ix^i\in R\llb x \rrb \mid a_0\in R^\times \right\}. 
\end{equation}
The ring of polynomials $R[x]$ is a subring of $R\llb x \rrb$. We will also consider the following subring 
\begin{equation}
R[x]^l = \{ fg^{-1} \in R\llb x \rrb \mid f,g\in R[x],\  g \in R[x]\cap R\llb x\rrb^\times  \}. 
\end{equation}
Elements of this ring are called rational power series.

For two subrings $S,T$ of a commutative ring we denote by $S\cdot T$ the subring consisting of elements of the form $s_1t_1+\dots+s_nt_n, s_i\in S,t_i\in T.$ We will be interested in the following subrings of the ring of power series on two variables 
\begin{equation}
  R[x]^l\cdot R[x]^l \ \subset \ R\llb x \rrb \cdot R \llb y \rrb   \  \subset  \  R\llb x,y \rrb.
\end{equation}
The subring  
$R\llb x\rrb\cdot R\llb y\rrb$
is known as the subring of series of finite rank. It has a subring $\Sym(R\llb x\rrb\cdot R\llb y\rrb)$ of power series $F(x,y)$ having a finite rank and satisfying $F(x,y) = F(y,x)$.

The aim of this section is to prove the following technical theorem about integral power series which will play a key role in the proof of the fact that $H_2( \widehat{F} )$ is not cotorsion.

\begin{theorem}
\label{th:power_series_cotorsion} 
Assume that $P(x,y) = U(x) - y V(x)\in \ZZ[x,y]$ is a polynomial satisfying $P(x,y) = \pm P(y,x)$. Assume moreover that $\frac{U(x)}{V(x)}\in \ZZ[x]^l\smallsetminus \ZZ$.
Then the group 
\begin{equation}
\frac{\ZZ\llb x \rrb \cdot \ZZ\llb y \rrb}{
(U(x)-V(x)y)\cdot \ZZ\llb x \rrb \cdot \ZZ\llb y \rrb + \ZZ[x]^l\cdot \ZZ[x]^l + \Sym(\ZZ\llb x\rrb\cdot \ZZ\llb y\rrb)
}
\end{equation}
is not cotorsion. 
\end{theorem}

\subsection{Power series of finite rank}
Let 
$A$ 
be a commutative ring and 
$R$ be its subring. 
Then $A$ has a natural structure of $R$-algebra. 
We consider the following subring in the ring of power series in one variable 
$x$ over $A$ 
\begin{equation}
A\cdot R\llb x \rrb \subseteq A\llb x \rrb, 
\end{equation}
which is the product of the subrings $A$ and $R\llb x \rrb$ of $A\llb x \rrb.$

\begin{lemma}\label{lemma:finite_rank}
Let $R\subseteq A$ be an extension of commutative rings and  $F=\sum f_ix^i$ be a power series from $A \llb x \rrb.$ Then the following statements are equivalent. 
\begin{enumerate}
\item $F\in A\cdot R\llb x\rrb;$ 
\item the elements $f_0,f_1,\dots $ lie in a finitely generated $R$-submodule of  $A$.
\end{enumerate}
\end{lemma}
\begin{proof} For a power series $F=\sum f_ix^i$ we denote by $\langle F \rangle \subseteq A $ the $R$-submodule of $A$ generated by $f_0, f_1,\dots.$ We also denote by $  \mathfrak{A} \subseteq A\llb x \rrb$ the set of power series such that $\langle F \rangle$ is contained in a finitely generated $R$-submodule of $A$. Note that $\langle F+G \rangle \subseteq \langle F \rangle + \langle G \rangle.$ It follows that $\mathfrak{A}$ is an additive subgroup of $A\llb x \rrb.$ If $F=aG$ for some  $a\in A, G \in R \llb x \rrb,$ then $\langle F \rangle$ is a submodule of $Ra,$ and hence, $F\in \mathfrak{A}.$ Therefore, using that $\mathfrak{A}$ is an additive subgroup, we obtain $A\cdot R \llb x \rrb \subseteq \mathfrak{A}.$ 

Prove that $\mathfrak{A}\subseteq A\cdot R \llb x \rrb.$  Take a power series  $F=\sum f_ix^i\in \mathfrak{A}.$ Then there exists a finite collection  $a_1,\dots, a_n\in A$ such that $f_i=\sum_{j=1}^n r_{i,j} a_j$ for some $r_{i,j}\in R.$ Therefore
$
F= \sum_{i=1}^\infty \sum_{j=1}^n  r_{i,j} a_j x^i  = \sum_{j=1}^n a_j G_j,    
$
where $G_j=\sum_{i=1}^\infty r_{i,j}x^i.$
\end{proof}

If $R=k$ is a field, then we can define the $k$-rank of a power series $F=\sum f_ix^i\in A\llb x \rrb$ as the dimension
\[
\rk_k(F)=\dim_k\spn_k\{ f_0,f_1,\dots \}.
\]

\begin{corollary} Let $k$ be a field, $A$ be a commutative $k$-algebra, and  $F=\sum f_ix^i$ be a power series from $ A \llb x \rrb.$ Then $F\in A\cdot k\llb x \rrb$ if and only if $\rk_k(F)< \infty.$
\end{corollary}

\begin{corollary} \label{cor:finite_rank}
Let $F(x,y)=\sum f_i(x)y^i$ be a power series from $\ZZ\llb x,y \rrb.$ Then $F\in \ZZ\llb x \rrb \cdot \ZZ \llb y \rrb  $ if and only if the abelian group generated by  $f_0(x),f_1(x),\dots \in \ZZ\llb x \rrb $ is finitely generated. 
\end{corollary}

\begin{lemma}\label{lemma:torsionfree}
The abelian group $\ZZ\llb x,y \rrb/ ( \ZZ\llb x \rrb \cdot \ZZ\llb y \rrb  )$ is torsion free. 
\end{lemma}
\begin{proof}
Corollary \ref{cor:finite_rank} implies that a power series $F=\sum_{i=0}^\infty f_i(x)y^i \in \ZZ\llb x,y \rrb$ lies in $\ZZ\llb x \rrb \cdot  \ZZ\llb y \rrb  $ if and only if the abelian group generated by $f_0(x),f_1(x),\dots $ is finitely generated. For any $n$ the abelian group generated by $f_0(x),f_1(x),\dots$ is isomorphic to the abelian group  generated by  $nf_0(x),nf_1(x),\dots$. Hence $nF \in \ZZ\llb x \rrb \cdot  \ZZ\llb y \rrb $ implies $F\in \ZZ\llb x \rrb \cdot  \ZZ\llb y \rrb.$ The assertion follows. 
\end{proof}

\subsection{Sieves}
\label{subsec:preliminary-lemmas}

In this subsection we assume that $k\subset K$ is a proper field extension and $A$ is a $K$-algebra
\begin{equation}
  k \subset K \subset A.
\end{equation}

\begin{definition}
  \label{defin:sieve}
  Let $n,d$ be positive integers and $F = \sum f_jx^j\in A\llb x\rrb$ be a formal power series. We say that $F$ has an $n,d$-\emph{sieve} if there exists $m\geq 0$ such that
  \begin{itemize}
    \item for any $l = 0,1,\dots,n$ and $i = 1,\dots,d-1$ we have $f_{m + ld + i} = 0$ and
    \item the coefficients $f_{m+d}, f_{m+2d},\dots, f_{m+nd}\in A$ are linearly independent over $K$.
  \end{itemize}
\end{definition}

As we will see below, the concept of a sieve provide a measurement that allows to indicate if a formal series belongs to the denominator in Theorem~\ref{th:power_series_cotorsion}. More formally, we have the following lemma, which is the central ingredient in the proof of the theorem.

\begin{lemma}
  \label{lemma:sieve-vs-rank}
  Assume that $\alpha,\beta\in K^\times$ are such that $1,\alpha\beta^{-1}, (\alpha\beta^{-1})^2,\dots$ are linearly independent over $k$ and $H,G\in A\llb x\rrb$ are two series such that
  \begin{enumerate}
    \item\label{enum:sieve-vs-rank-rkH} $k$-rank of $H$ is at most $d-1$;
    \item\label{enum:sieve-vs-rank-rkG} $K$-rank of $G$ is at most $n-1$.
  \end{enumerate}
  Then the series $F$ defined as
  \[
    F = (\beta-\alpha x)\cdot H + G
  \]
  cannot have an $n,d$-sieve.
\end{lemma}
\begin{proof}
  Without loss of generality we can assume that $\beta=1$, so that
  \begin{equation}
    \label{eq:sieve-vs-rank1}
    F = (1-\alpha x)\cdot H + G.
  \end{equation}
  The element $(1-\alpha x)\in K\llb x\rrb$ is invertible with the inverse given by $(1-\alpha x)^{-1} = \sum\limits_{j\geq 0}\alpha^jx^j$. Multiplying~\eqref{eq:sieve-vs-rank1} by this inverse we obtain
 \begin{equation}
  \left(\sum_{j=0}^\infty \alpha^j x^j\right)\cdot F =  H + \left(\sum_{j=0}^\infty \alpha^j x^j\right)\cdot G.   
 \end{equation}
Substituting $F = \sum\limits_{j\geq 0}f_jx^j$, $H = \sum\limits_{j \geq 0}h_jx^j$ and $G = \sum\limits_{j\geq 0}g_jx^j$ one gets 
  \begin{equation}
    \label{eq:sieve-vs-rank2}
    \sum_{j\geq 0}\bigl(\sum_{s+t=j}f_s \alpha^t\bigr) \cdot x^j = \sum_{j\geq 0}\bigl( h_j + \sum_{s+t = j}g_s \alpha^t \bigr)\cdot x^j.
  \end{equation}
  Let $V\subset A$ denote the $K$-span of the set $g_0,g_1,g_2,\dots\in A$. Due to the assumption~\ref{enum:sieve-vs-rank-rkG} of the lemma we have $\dim_k V \leq n-1$.

  Recall that $\alpha\in K$, hence $\sum_{s+t = j}g_s \alpha^t\in V$ for any $j$. Using this observation and~\eqref{eq:sieve-vs-rank2} we find out the following relation for any $j\geq 0$:
  \begin{equation}
    \label{eq:sieve-vs-rank3}
    \sum_{s+t=j}f_s \alpha^t - h_j\in V.
  \end{equation}

  Assume now that $F$ has an $n,d$-sieve and let $m$ be the corresponding integer from Definition~\ref{defin:sieve}. Set $a = \sum_{s+t = m}f_s \alpha^t$. The definition of the sieve together with~\eqref{eq:sieve-vs-rank3} for $j = m,m+1,\dots,m+n^2+n-1$ boil down to the following:
  \begin{equation}
    \label{eq:sieve-vs-rank4}
    \begin{split}
     & \alpha^i\cdot a - h_{m+i}\in V\quad \forall i = 0,1,\dots d-1,\\
     & \alpha^{i+d}\cdot a + \alpha^i\cdot f_{m+d} - h_{m+d+i}\in V\quad \forall i = 0,1,\dots d-1,\\
     & \vdots\\
     & \alpha^{i+nd}\cdot a + \alpha^{i+(n-1)d}\cdot f_{m+d} + \ldots + \alpha^if_{m+nd} - h_{m+nd+i}\in V \quad \forall i = 0,1,\dots d-1.
    \end{split}
  \end{equation}
  Due to the assumption~\ref{enum:sieve-vs-rank-rkH} of the lemma we have that for any $m'$ the coefficients $h_{m'}, h_{m'+1},\dots,h_{m'+d-1}$ are linearly dependent over $k$. For any $l = 0,1,\dots, n$ let $b_0h_{m+ld} + \ldots + b_{d-1}h_{m+ld+d-1}=0$ be a non-trivial relation with $b_0,b_1,\dots,b_{d-1}\in k$ and set
  \[
    \lambda_l = b_0+ \alpha b_1 + \alpha^2b_2 + \ldots + \alpha^{d-1}b_{d-1}.
  \]
  Due to the assumption on $\alpha,\beta$ (recall that we set $\beta = 1$) we have $\lambda_l\in K^\times$ for any $l$. Moreover, relations~\eqref{eq:sieve-vs-rank4} imply (line-by-line) relations
  \begin{equation}
    \label{eq:sieve-vs-rank5}
    \begin{split}
      & \lambda_0\cdot a\in V,\\
      & \alpha^d\lambda_1\cdot a + \lambda_1\cdot f_{m+n}\in V,\\
      & \vdots\\
      & \alpha^{nd}\lambda_n\cdot a + \alpha^{(n-1)d}\lambda_n\cdot f_{m+d} + \ldots + \lambda_n\cdot f_{m+nd}\in V.
    \end{split}
  \end{equation}
  This immediately implies that $f_{m+d}, f_{m+2d},\dots, f_{m+nd}\in V$. But on the other hand $f_{m+d}, f_{m+2d},\dots, f_{m+nd}$ are linearly independent over $K$ due to Definition~\ref{defin:sieve} of an $n,d$-sieve, while $\dim_k V\leq n-1$ by the construction. Thus we obtain a contradiction.
\end{proof}

\subsection{Existence of a divisible series with an arbitrary sieve} Let $p$ be a prime number and $\mathbb{F}_p$ be a field of order $p$.  Recall that $\mathbb{F}_p(x)$ denotes the field of rational functions over $\mathbb{F}_p$ and $\mathbb{F}_p\llp x\rrp$ denotes the \emph{field} of Laurent series of one variable $x$. In this section we will always apply Definition~\ref{defin:sieve} of an $n,d$-sieve for 
\begin{equation}
k=\mathbb{F}_p, \hspace{1cm} K=\mathbb{F}_p(x), \hspace{1cm}
A = \mathbb{F}_p\llp x\rrp
\end{equation}
and for a power series in two variables from
\begin{equation}
\mathbb{F}_p\llb x,y \rrb \subset  (\mathbb{F}_p\llp x\rrp) \llb y \rrb.
\end{equation}
First, we need a tool to verify a linear independence over $K$.

\begin{lemma}
\label{lemma:linearly_indep_new}
Fix a prime number $p$. Let $\mathcal{A}$ be an arbitrary set and $g_\alpha(x) = \sum_{i\in \mathbb Z}a_{\alpha,i}x^i\in \mathbb{F}_p\llp x\rrp$ be a family of power series enumerated by elements $\alpha\in \mathcal{A}$. Assume that the following property holds for this family: for any $N>0$ and $\alpha\in \mathcal{A}$ there exists $m = m(\alpha,N)\in \mathbb Z$ such that $a_{\alpha,m}\neq 0$, $a_{\beta, m} = 0$ for any $\beta\neq \alpha$ and moreover $a_{\beta,m+i} = 0$ for any $\beta$ (including $\beta = \alpha$) and $0<|i|<N$. Then the family $\{g_\alpha\}_{\alpha\in \mathcal{A}}$ is lineraly independent over $\mathbb{F}_p(x)$.
\end{lemma}
\begin{proof}
Assume the contrary. Then one can find $\alpha_1,\dots,\alpha_n\in \mathcal{A}$ and non-zero rational functions $r_1(x),\dots,r_n(x)\in \mathbb{F}_p(x)$ such that
\[
r_1(x)g_{\alpha_1}(x) + r_2(x)g_{\alpha_2}(x) + \ldots + r_n(x)g_{\alpha_n}(x) = 0.
\]
Note that one can assume that all $r_i$'s are polynomials (if not, just multiply the equation above by the common denominator of $r_i$'s). Let $d_i = \deg r_i$ and $N>\max\{d_1,\dots, d_n\}$. Let $m = m(\alpha_n, N)$ be the index promised by the assumption made in the lemma. Let us also write $r_n(x) = bx^{d_n} + \ldots$ where $b\neq 0$ is the leading coefficient of $r_n$. Then it is easy to see that the formal power series $r_1(x)g_{\alpha_1}(x) + r_2(x)g_{\alpha_2}(x) + \ldots + r_n(x)g_{\alpha_n(x)}$ has its coefficient at $x^{m + d_n}$ equal to $a_{\alpha_n, m}\cdot b\neq 0$ which contradicts with the fact that this series vanishes.
\end{proof}

Now we can present the main construction that will be used as an obstruction for the group from Theorem~\ref{th:power_series_cotorsion} to be a cotorsion.

\begin{lemma}
  \label{lemma:existence-of-sieve}
  There exists a power series $F\in \ZZ\llb x,y\rrb$ such that:
  \begin{itemize}
    \item the element of $\frac{\ZZ\llb x,y\rrb}{\ZZ\llb x \rrb\cdot \ZZ\llb y\rrb}$ defined by $F$ is divisible by any prime  $p$;
    \item for any prime $p$ and any $d\geq p$ the power series $F(x,y) - F(y,x)$ considered as an element of $(\mathbb{F}_p\llp x\rrp)\llb y\rrb$ has an $p,d$-sieve.
  \end{itemize}
\end{lemma}
\begin{proof}
We construct the series $F(x,y)$ explicitly. For any integers $i,k$, let us define 
\[
    s_{i,k} = 3^i + (k+1)i,\qquad t_{i,k} = 3^i - (k+1)i
\]
and set
\begin{multline}
    \label{eq:formula_for_F}
    F(x,y) = \sum_{i\geq 1}\sum_{j\geq 1}\sum_{k = 0}^{\min(i,j)}k! \cdot x^{s_{i,k}}y^{t_{j,k}} =\\
     = \sum_{j\geq 1}\sum_{k = 0}^j k!\cdot y^{t_{j,k}}\sum_{i\geq k}x^{s_{i,k}} = \sum_{i\geq 1}\sum_{k = 0}^i k!\cdot x^{s_{i,k}}\sum_{j\geq k}y^{t_{j,k}}.
\end{multline}
Put $g_k(x) = k!\cdot \sum_{i\geq k} x^{s_{i,k}}$ and $h_k(x) = k!\cdot \sum_{i\geq k} x^{t_{i,k}}$. The second equality in~\eqref{eq:formula_for_F} reads as 
\[
F(x,y) = \sum_{j\geq 1}\sum_{k = 0}^j y^{t_{j,k}}\cdot g_k(x).
\]
By the construction, for any prime $p$ the series $g_k(x)\mod p$ vanishes if $k\geq p$, which implies by Lemma~\ref{lemma:finite_rank} that $F(x,y)$ is equal to a series of a finite rank modulo $p$, so that $F$ satisfies the first part of the lemma.

Moreover, by the second and the third equality from~\eqref{eq:formula_for_F} one has
\[
    F(x,y) - F(y,x) = \sum_{i\geq 1}\sum_{k = 0}^i \big( y^{t_{i,k}}\cdot g_k(x) - y^{s_{i,k}}\cdot h_k(x) \big).
\]
Fix a prime number $p$ and $d\geq p$. Note that we have \[\ldots < s_{d-1,d-1} < t_{d,d} < t_{d,d-1} < \ldots < t_{d,0} < s_{d,0} < s_{d,1} < \ldots < s_{d,d} < t_{d+1,d+1} < \ldots\] and moreover $t_{d,k} - t_{d,k+1} = d$ for any $k$ and $t_{d,0} + d \leq s_{d,0}$. It follows that the element in $(\mathbb{F}_p\llp x\rrp)\llb y\rrb$ defined by $F(x,y) - F(y,x)$ has an $p,d$-sieve if the elements in $\mathbb{F}_p\llp x\rrp$ defined by $h_0(x),h_1(x),\dots,h_{p-1}(x)$ are linearly independent over $\mathbb{F}_p(x)$. But it is straightforward to check that these elements satisfy assumptions of Lemma~\ref{lemma:linearly_indep_new}, so they are independent by this lemma.
\end{proof}

\begin{lemma}
\label{lemma:powers-of-UoverVmodp}
Assume that $U(x), V(x)\in \ZZ[x]$ are polynomials satisfying conditions of Theorem~\ref{th:power_series_cotorsion}. Then there exists a $p_0$ such that for any prime number $p\geq p_0$ the elements $\alpha,\beta \in \mathbb{F}_p(x)$ defined as $\alpha = U(x)\mod p$ and $\beta = V(x)\mod p$ satisfy conditions of Lemma~\ref{lemma:sieve-vs-rank} with $k = \mathbb{F}_p$ and $K = \mathbb{F}_p(x)$.
\end{lemma}
\begin{proof}
By assumptions of Theorem~\ref{th:power_series_cotorsion} we have $\frac{U(x)}{V(x)}\in \ZZ[x]^l\smallsetminus \ZZ$ which implies that $\alpha\beta^{-1}\in K\smallsetminus k$ for any $p$ large enough. Fix such a $p$ and assume that $1,\alpha\beta^{-1},(\alpha\beta^{-1})^2,\dots$ are linearly dependent over $k$. This implies that there exist an integer $n>0$ and  $c_0,c_1,\dots, c_n\in k$ such that 
\[
    c_0\alpha^n + c_1\alpha^{n-1}\beta + \ldots + c_n\beta^n = 0
\]
and $c_0c_n\neq 0$. But this immediately implies that polynomials $\alpha,\beta\in \mathbb{F}_p[x]\subset K$ divide each other as elements of the ring $ \mathbb{F}_p[x]$ which is possible only if $\alpha\beta^{-1}\in k$.
\end{proof}

\subsection{Proof of Theorem \ref{th:power_series_cotorsion}} 
In order to prove that a group $A$ is not cotorsion, it is sufficient to construct a non-split short exact sequence $0\to A\to B\to C\to 0$ with torsion free $C.$ Set

\begin{equation}
S=(U(x)-V(x)y)\cdot \ZZ\llb x \rrb \cdot \ZZ\llb y \rrb + \ZZ[x]^l\cdot \ZZ[x]^l + \Sym(\ZZ\llb x\rrb \cdot \ZZ\llb y\rrb).    
\end{equation}
We want to prove that the group
$(\ZZ\llb x \rrb \cdot \ZZ\llb y \rrb)/S$
is not cotorsion. Consider the following short exact sequence 
\begin{equation}
0 \longrightarrow \frac{\ZZ\llb x \rrb \cdot \ZZ\llb y \rrb}{S} \longrightarrow \frac{\ZZ\llb x, y \rrb}{S} \longrightarrow \frac{\ZZ\llb x,y \rrb}{ \ZZ\llb x \rrb \cdot \ZZ\llb y \rrb } \longrightarrow 0.    
\end{equation}
 By Lemma \ref{lemma:torsionfree} the group $\frac{\ZZ\llb x,y \rrb}{ \ZZ\llb x \rrb \cdot \ZZ\llb y \rrb }$ is torsion free. Therefore the theorem follows from Proposition~\ref{prop:the-morphism-dont-split}.

\begin{proposition}
  \label{prop:the-morphism-dont-split} Under assumptions of Theorem \ref{th:power_series_cotorsion}  the following epimorphism of abelian groups does not split:
  \[
    \frac{\ZZ\llb x,y\rrb}{(V(x) - U(x)y)\cdot\ZZ\llb x\rrb\cdot \ZZ\llb y\rrb + \ZZ[x]^l\cdot \ZZ[ y]^l + \Sym(\ZZ\llb x\rrb \cdot \ZZ\llb y\rrb)} \twoheadrightarrow \frac{\ZZ\llb x,y\rrb}{\ZZ\llb x \rrb\cdot  \ZZ\llb y\rrb}.
  \]
\end{proposition}
\begin{proof}
Let $F(x,y)\in \ZZ\llb x,y\rrb$ be the series defined in Lemma~\ref{lemma:existence-of-sieve} and let $[F]\in \frac{\ZZ\llb x,y\rrb}{\ZZ\llb x \rrb\cdot \ZZ\llb y\rrb}$ denote the corresponding element. The proposition will be proved if we show that for any lift of $[F]$ under the morphism from the proposition there exists a prime $p$ that does \emph{not} divide this lift.

Any such lift can be represented by a series of the form $F(x,y) - G_1(x,y)$, where $G_1(x,y)$ has a finite rank. Fix such a lift. Note that $G_1(x,y) - G_1(y,x)$ has a finite rank too. Write $G_1(x,y) - G_1(y,x) = \sum\limits_{j\geq 0}g_j(x)y^j\in \ZZ\llb x,y\rrb$. By Corollary \ref{cor:finite_rank} the sequence $g_0(x), g_1(x),\dots$ spans a finitely generated abelian group of some rank  $n_0.$

Let $p_0$ be the constant from Lemma~\ref{lemma:powers-of-UoverVmodp} and $p\geq \max\{n_0+2,p_0\}$ be a prime number. The fact that $F(x,y)-G_1(x,y)$ is divisible by $p$ as an element of $\frac{\ZZ\llb x,y\rrb}{(V(x) - U(x)y)\cdot\ZZ\llb x\rrb\cdot \ZZ\llb y\rrb + \ZZ[x]^l\cdot \ZZ[ y]^l +  \Sym(\ZZ\llb x\rrb \cdot \ZZ\llb y\rrb)}$ means that one can find some series $H_0(x,y)\in \ZZ\llb x\rrb\cdot \ZZ\llb y\rrb$, $G_2(x,y)\in \ZZ[x]^l\cdot \ZZ[x]^l$ and $S(x,y)\in \Sym(\ZZ\llb x\rrb \cdot \ZZ\llb y\rrb)$ such that \[F(x,y) - (V(x)-U(x)y)\cdot H_0(x,y)- G_1(x,y) - G_2(x,y) - S(x,y)\] is divisible by $p$ as an element of $\ZZ\llb x,y\rrb$. This can be reformulated as 
\begin{multline}
  \label{eq:div-by-p1}
  \text{\textit{one has in }}\mathbb{F}_p\llb x,y\rrb:\\  F(x,y)= \\= (V(x)-U(x)y)\cdot H_0(x,y)+ G_1(x,y) + G_2(x,y) + S(x,y).
\end{multline}
Note that there is a natural embedding $\mathbb{F}_p\llb x,y\rrb\hookrightarrow (\mathbb{F}_p\llp x\rrp)\llb y\rrb$, in particular, we can assume that~\eqref{eq:div-by-p1} holds in $(\mathbb{F}_p\llp x\rrp)\llb y\rrb$. Using that $S(x,y) = S(y,x)$ by the definition and the polynomial $P(x,y) = (V(x)-U(x)y)$ is either symmetric or anti-symmetric we find out that 
\begin{multline}
  \label{eq:div-by-p2}
  \text{\textit{one has in }}(\mathbb{F}_p\llp x\rrp)\llb y\rrb:\\
  F(x,y) - F(y,x)= \\= (V(x)-U(x)y)\cdot \widetilde{H}_0(x,y)+ \widetilde{G}_1(x,y) + \widetilde{G}_2(x,y)
\end{multline}
where $\widetilde{H}_0(x,y) = H_0(x,y) \pm H_0(y,x)$ and $\widetilde{G}_{1,2}(x,y) = G_{1,2}(x,y) - G_{1,2}(y,x)$.
Define now
\begin{itemize}
  \item $A = \mathbb{F}_p\llp x\rrp$, $K=\mathbb{F}_p(x),$ $k=\mathbb{F}_p.$
  \item $\alpha\in K^\times$ corresponds to $U(x)$ and $\beta\in K^\times$ corresponds to $V(x)$ (note that the fact that $\alpha,\beta\neq 0$ is guaranteed by the fact that $p\geq p_0$ and Lemma~\ref{lemma:powers-of-UoverVmodp});
  \item $G(y)\in A\llb y\rrb$ corresponding to $\widetilde{G}_1(x,y) + \widetilde{G}_2(x,y)$ and $H(y) \in A\llb y\rrb$ corresponding to $\widetilde{H}_0(x,y)$;
  \item $d\geq p$ be any integer greater than the $\mathbb{F}_p$-rank of $H(y)$ and $n_1$ to be the $K$-rank of $G(y)$; note that $d<\infty$ and $n_1\leq n_0+1$ by the construction, since the coefficients of $\widetilde{G}_2(x,y)$ considered as a series of $y$, are rational functions of $x$.
\end{itemize}
Now it is straightforward that $H,G$ satisfy assumptions of Lemma~\ref{lemma:sieve-vs-rank} with the chosen $d$ and and $n=p$ since $p\geq n_0+2\geq n_1+1$. Moreover, $\alpha$ and $\beta$ satisfy assumptions of Lemma~\ref{lemma:sieve-vs-rank} due to Lemma~\ref{lemma:powers-of-UoverVmodp} and the choice of $p$. Therefore, with all this setup we can apply Lemma~\ref{lemma:sieve-vs-rank} and conclude from~\eqref{eq:div-by-p2} that the element in $(\mathbb{F}_p\llp x\rrp)\llb y\rrb$ defined by $F(x,y) - F(y,x)$ cannot have an $p,d$-sieve. But it has $p,d$-sieve by the construction (recall that $F(x,y)$ is the series from Lemma~\ref{lemma:existence-of-sieve}), so we obtain a contradiction.
\end{proof}

\subsection{Some related results} 

In this subsection we prove some results related to the topic of this section but which are unnecessary for the main topic of the article (Theorem A and Theorem B). 

We set 
\begin{equation}
P=\prod_{i=1}^\infty \ZZ, \hspace{1cm} S=\bigoplus_{i=1}^\infty \ZZ.
\end{equation}
The group $P$ is known as Baer–Specker group. By the Hulanicki theorem the group $P/S$ is cotorsion. Moreover, any at most continuous cotorsion group is a quotient of $P/S$ \cite[Th.7]{goldsmith2004cotorsion}.

\begin{proposition}
The groups 
\begin{equation}
\frac{P\otimes P}{S\otimes S}, \hspace{1cm}
\frac{\Lambda^2 P}{\Lambda^2 S}    
\end{equation}
are not cotorsion.
\end{proposition}
\begin{proof}
The group $\frac{\Lambda^2 P}{\Lambda^2 S} $ is a quotient of $\frac{P\otimes P}{S\otimes S}.$ Therefore, it is enough to prove the statement for $\frac{\Lambda^2 P}{\Lambda^2 S}.$ Note that there are isomorphisms of abelian groups $P\cong \ZZ\llb x \rrb$ and $S\cong \ZZ[x]$ that respect inclusions. So it is enough to prove that $\frac{\Lambda^2 \ZZ\llb x \rrb}{\Lambda^2 \ZZ[x]}$ is not cotorsion. Consider the homomorphism 
\begin{equation}
    \theta : \ZZ \llb x\rrb\otimes \ZZ \llb x \rrb \longrightarrow \ZZ\llb x,y \rrb, \hspace{1cm} f\otimes g \mapsto f(x)g(y). 
\end{equation}
The image of $\theta$ is $\ZZ\llb x \rrb\cdot \ZZ\llb y \rrb $ and the image of the subgroup generated by the elements $f\otimes f$ lies in ${\sf Sym}(\ZZ\llb x \rrb \cdot \ZZ\llb y \rrb).$ Therefore we have an epimorphism 
\begin{equation}
\Lambda^2 \ZZ\llb x \rrb \epi (\ZZ\llb x \rrb \cdot \ZZ\llb y \rrb)/ {\sf Sym}( \ZZ\llb x \rrb \cdot \ZZ\llb y \rrb ).  
\end{equation}
This epimorphism induces an epimorphism 
\begin{equation}
\frac{\Lambda^2 \ZZ\llb x \rrb}{\Lambda^2 \ZZ[x] } \epi \frac{\ZZ\llb x \rrb \cdot \ZZ\llb y \rrb}{\ZZ[x]^l\cdot \ZZ[y]^l+ {\sf Sym}( \ZZ\llb x \rrb \cdot \ZZ\llb y \rrb )}.
\end{equation}
By Theorem \ref{th:power_series_cotorsion} a quotient of the last group is not cotorsion. The assertion follows. 
\end{proof}

For a prime number $p$ we set
\begin{equation}
    P_p = \prod_{i=1}^\infty \ZZ_p,
\end{equation}
where $\ZZ_p=\varprojlim \ZZ/p^i$ is the group of $p$-adic integers. The group $P_p$ is cotorsion.
\begin{proposition}\label{prop:P_p}
The group $P_p\otimes P_p$ is not cotorsion.
\end{proposition}
\begin{proof}
Note that there is an isomorphism $P_p=\ZZ_p\llb x \rrb.$  The group $\ZZ_p\llb x \rrb \otimes \ZZ_p\llb x \rrb $ maps onto the group $\ZZ_p\llb x \rrb \cdot \ZZ_p\llb y \rrb \subseteq \ZZ\llb x,y \rrb.$ So it is enough to prove that $\ZZ_p\llb x \rrb \cdot \ZZ_p\llb y \rrb$ is not cotorsion. Consider the short exact sequence
\begin{equation}\label{eq:P_p} 0\longrightarrow
\ZZ_p\llb x \rrb \cdot \ZZ_p\llb y \rrb \longrightarrow \ZZ_p\llb x,y \rrb \longrightarrow \frac{\ZZ_p\llb x,y\rrb }{\ZZ_p\llb x \rrb \cdot \ZZ_p\llb y \rrb } \longrightarrow 0.   
\end{equation}
So it is sufficient to prove that: (1) $\frac{\ZZ_p\llb x,y\rrb }{\ZZ_p\llb x \rrb \cdot \ZZ_p\llb y \rrb }$ is torsion free and (2) the short exact sequence \eqref{eq:P_p} does not split.

(1) By Lemma \ref{lemma:finite_rank} a power series $f(x,y)=\sum a_i(x)y^i \in \ZZ_p\llb x,y \rrb $ lies in $\ZZ_p\llb x \rrb \cdot \ZZ_p\llb y\rrb $ if and only if the $\ZZ_p$-submodule generated by $a_0(x),a_1(x),\dots $ is finitely generated. On the other hand the submodule generated by  $na_0(x), na_1(x),\dots$ is isomorphic to the submodule generated by $a_0(x),a_1(x),\dots$. Therefore $nf\in \ZZ_p\llb x \rrb \cdot \ZZ_p\llb y\rrb$ if and only if $f\in \ZZ_p\llb x \rrb \cdot \ZZ_p\llb y\rrb.$ 

(2) Consider the power series $f(x,y)=\sum_{i=0}^\infty p^ix^iy^i.$ Its image in $\frac{\ZZ_p\llb x,y\rrb }{\ZZ_p\llb x \rrb \cdot \ZZ_p\llb y \rrb }$ is nontrivial (because the $\ZZ_p$-submodule generated by $1,px,p^2x^2,\dots$ is not finitely generated) and it is divisible by any power of $p.$ On the other hand there is no a non-trivial element divisible by all powers of $p$ in $\ZZ_p\llb x,y \rrb.$ The assertion follows.
\end{proof}

\begin{remark}
Proposition \ref{prop:P_p} shows that the tensor product of cotorsion groups is not necessarily cotorsion.
\end{remark}

\begin{proposition}
There exists a family of continuous cardinality  $(g_\alpha)_{\alpha\in 2^{\aleph_0}}$ of power series  $g_\alpha\in \ZZ\llb x \rrb$  such that its image in  $\mathbb{F}_p\llp x \rrp $ is linearly independent over $\mathbb{F}_p(x)$ for any prime $p.$
\end{proposition}
\begin{proof}
Consider a family $(X_r)_{r\in \mathbb{R}}$ of subsets of natural numbers $X_r\subseteq \mathbb{N}$ indexed by real numbers $r\in \mathbb{R}$ such that for any $s<r$ we have $X_s\subseteq X_r$ and $X_r\setminus X_s$ is infinite. For example, we can renumber all rational numbers $\{a_1,a_2,\dots \}=\mathbb{Q}$ and define $X_r =\{ n \mid a_n< r \}.$ 
Then we define $g_r = \sum_{n\in X_r} x^{2^n}$ and consider the family $(g_r)_{r\in \mathbb{R}}.$ 
Fix some prime $p$ and denote by $\bar g_r$ the image of $g_r$ in $\mathbb{F}_p\llp x \rrp$ and prove that $(\bar g_r)_{r\in \mathbb{R}}$ is linearly independent over $\mathbb{F}_p(x).$ 

For, pick up an arbitrary finite collection $\bar g_{r_1},\bar g_{r_2},\dots,\bar g_{r_n}$ enumerated such that $r_1<r_2<\ldots<r_n$. Define $h_1 = \bar g_{r_1}$ and $h_2 = \bar g_{r_2} - \bar g_{r_1},\dots, h_n = \bar g_{r_n} - \bar g_{r_{n-1}}$. Then it is straightforward to check that $h_1,h_2,\dots, h_n$ satisfy the assumption of Lemma~\ref{lemma:linearly_indep_new}, thus they are linearly independent and so are $\bar g_{r_1},\dots, \bar g_{r_n}$.
\end{proof}

\section{Integral lamplighter group}

\subsection{Definition of the lamplighter group}
The classical lamplighter group can be defined as the restricted wreath product $\ZZ/2 \wr \ZZ.$ Here we consider its integral version 
\begin{equation}
L=\ZZ \wr \ZZ.    
\end{equation}
Denote by $C=\langle t \rangle$ the infinite cyclic group written multiplicatively generated by an element $t.$ Note that $\bigoplus_{w\in C} \ZZ = \ZZ[C]=\ZZ[t,t^{-1}].$ Then
\begin{equation}
L = C \ltimes \ZZ[t,t^{-1}], 
\end{equation}
where $C$ acts on $\ZZ[t,t^{-1}]$ by the multiplication. 
This group has the following presentation
\begin{equation}
L=\langle t,a \mid [a, a^{t^m}]=1 , \  m\in \ZZ  \rangle
\end{equation}
The main goal of this section is to prove that the cokernel of the map
\begin{equation}
H_2(L)\longrightarrow H_2(\widehat{L})    
\end{equation}
is not cotorsion (Theorem \ref{th:double_lamplighter}).

\subsection{The completion}

Denote by $\ZZ\llb x\rrb$ the ring of power series and consider the ring homomorphism 
\begin{equation}
\varphi : \ZZ[t,t^{-1}] \longrightarrow \ZZ\llb x \rrb, \hspace{1cm} \varphi(t)=1+x. 
\end{equation}
We will consider the ring $\ZZ\llb x \rrb$ as a module over  
$\ZZ[t,t^{-1}]$
with the action induced by this homomorphism. 

\begin{lemma}\label{lemma:Lcompletion}
There is an an isomorphism 
\begin{equation}
\widehat{L}\cong C \ltimes  \ZZ\llb x\rrb,
\end{equation}
and the map $L\to \widehat{L}$ is induced by $\varphi.$
\end{lemma}
\begin{proof} It follows from
\cite[Prop.4.7]{ivanov2016problem} that $\widehat{L}\cong C\ltimes {\ZZ[t,t^{-1}]}^\wedge_I,$ where ${\ZZ[t,t^{-1}]}^\wedge_I= \varprojlim \ZZ[t,t^{-1}]/((t-1)^n).$ It is easy to see that there is an isomorphism $\ZZ[x]/(x^n)\cong \ZZ[t,t^{-1}]/((t-1)^n), x\mapsto t-1.$ Then $\ZZ[t,t^{-1}]^\wedge_I=\varprojlim \ZZ[x]/(x^n) = \ZZ\llb x \rrb.$ The assertion follows. 
\end{proof}

\subsection{The second homology}

For any $C$-module $M$ we denote by $M_C=H_0(C,M)$ its module of $C$-coinvarians and by $M^C=H^0(C,M)$ the module of $C$-invarians. We also denote by $\Lambda^2 M$ its exterior square over $\ZZ$ considered as a module over $C$ with the diagonal action: $t(m_1\wedge m_2) = tm_1 \wedge tm_2.$

\begin{lemma}\label{lemma:homologyOfCltimesU} 
For any $C$-group $U$ and any $n$ there is a short exact sequence 
\begin{equation}
0 \longrightarrow H_n(U)_C \longrightarrow H_n(C\ltimes U) \longrightarrow H_{n-1}(U)^C \longrightarrow 0.    
\end{equation}
\end{lemma}
\begin{proof}
The spectral sequence of the extension $1\to U\to C\ltimes U \to C \to 1$ consists of two columns: $H_0(C,H_n(M))=
H_n(M)_C$ and $H_1(C,H_n(M))=H^0(C,H_n(M))=H_n(M)^C.$ The assertion follows. 
\end{proof}

\begin{lemma}\label{lemma:secondHomologyOfCltimesM}
For any $C$-module $M$ such that $M^C=0$ there is a natural isomorphism 
\begin{equation}
H_2(C\ltimes M)\cong (\Lambda^2 M)_C.    
\end{equation}
\end{lemma}
\begin{proof} It follows from Lemma 
\ref{lemma:homologyOfCltimesU} for $n=2$ and the the fact that the second homology of an abelian group is naturally isomorphic to its exterior square.
\end{proof}

\begin{theorem}\label{th:double_lamplighter}
The cokernel of the map 
\begin{equation}\label{eq:H_2D}
H_2(L) \longrightarrow H_2(\widehat{L})
\end{equation}
is not cotorsion. 
\end{theorem}
\begin{proof}
By Lemma \ref{lemma:secondHomologyOfCltimesM} and Lemma \ref{lemma:Lcompletion} we have isomorphisms 
\begin{equation}
H_2(L)\cong (\Lambda^2 \ZZ[t,t^{-1}])_C, \hspace{1cm}
H_2(\widehat{L}) \cong (\Lambda^2 \ZZ[x])_C
\end{equation}
and the map $H_2(L)\to H_2(\widehat{L})$ is isomorphic to the map 
$(\Lambda^2 \varphi)_C.$
Then the cokernel 
\begin{equation}
A:={\sf Coker}\left((\Lambda^2 \varphi)_C :  ( \Lambda^2 \ZZ[t,t^{-1}])_C \longrightarrow (\Lambda^2 \ZZ\llb x \rrb)_C \right)
\end{equation}
is isomorphic to the  cokernel of
\eqref{eq:H_2D}. Consider a ring homomorphism 
\begin{equation}
\theta:\ZZ\llb x \rrb^{\otimes 2}\to \ZZ\llb x,y \rrb, \hspace{1cm}  \theta(f\otimes g)=f(x)g(y).  
\end{equation}
If we define an action of $C$ on $\ZZ\llb x,y \rrb$ by multiplication on the polynomial \mbox{$(1+x)(1+y),$} then $\theta$ becomes a homomorphism of $C$-modules.
Note that ${\rm Im}(\theta)=\ZZ\llb x \rrb\cdot \ZZ \llb y \rrb.$ Moreover, if we denote by $D$ the subgroup of $\ZZ\llb x \rrb^{\otimes 2}$ generated by the elements of the form $f\otimes f,$ where $f\in \ZZ\llb x \rrb,$ we obtain $\theta(D)\subseteq {\sf Sym}(\ZZ\llb x \rrb\cdot \ZZ \llb y \rrb).$ So  $\theta$ induces an epimorphism 
\begin{equation}
\theta': \Lambda^2 \ZZ\llb x \rrb  \epi  
\frac{\ZZ\llb x \rrb \cdot \ZZ \llb y \rrb}{{\sf Sym}(  \ZZ\llb x \rrb \cdot \ZZ \llb y \rrb)}.
\end{equation}
Since $(t-1) \cdot \ZZ\llb x \rrb \cdot \ZZ\llb y \rrb = (x+y+xy) \cdot  \ZZ\llb x \rrb \cdot \ZZ\llb y \rrb,$ the map $\theta'$ induces an epimorphism
\begin{equation}
(\Lambda^2 \ZZ\llb x \rrb)_C \epi \frac{\ZZ\llb x \rrb \cdot \ZZ\llb y \rrb }{(x+y+xy) \cdot \ZZ\llb x \rrb \cdot \ZZ\llb y \rrb +{\sf Sym}(\ZZ[x]\cdot \ZZ[y])}.
\end{equation}

The image of the map $\theta\circ \varphi^{\otimes 2}:  \ZZ[t,t^{-1}]^{\otimes 2} \longrightarrow \ZZ\llb x,y \rrb$  is a subring of $\ZZ\llb x,y\rrb$ generated by $x,y, (1+x)^{-1},(1+y)^{-1}.$ Then we obtain a well defined epimorphism
\begin{equation}
A \epi \frac{\ZZ\llb x \rrb \cdot \ZZ\llb y \rrb }{(x+y+xy) \cdot \ZZ\llb x \rrb \cdot \ZZ\llb y \rrb + \ZZ[x]^l\cdot \ZZ[y]^l + {\sf Sym}(\ZZ[x]\cdot \ZZ[y])  }.    
\end{equation}
By Theorem \ref{th:power_series_cotorsion} the image of this epimorphism is not cotosion. It follows that $A$ is not cotorsion.  
\end{proof}

\section{Proof of Theorem A}
\label{sec:proof_of_TheoremA}
We prove that for a free group $F$ of rank at least $2$ the group $H_2(\widehat{F})$ is not cotorsion. 

The lamplighter group $L$ is a $2$-generated group. Then for any free group $F$ of rank $\geq 2$ there is an epimorphism  $F\epi L.$ By Theorem B we obtain that 
\begin{equation}
C:={\rm Coker}( H_2(\widehat{F}) \oplus H_2(L) \longrightarrow H_2( \widehat{L} ) ) 
\end{equation}
is cotorsion. On the other hand by Theorem \ref{th:double_lamplighter} the group 
\begin{equation}
 B:={\rm Coker}( H_2(L) \longrightarrow H_2(\widehat{L}))   
\end{equation}
is not cotorsion. Consider the exact sequence
\begin{equation}
H_2( \widehat{F} ) \overset{\varphi}\longrightarrow B \longrightarrow C \longrightarrow 0.
\end{equation}
 We claim that the group ${\rm Im}(\varphi)$ is not cotorsion. Indeed, $B$ is an extension of a cotorsion group $C$ by ${\rm Im}(\varphi).$ If ${\rm Im}(\varphi)$ was cotorsion, $B$ would be also cotorsion. Then ${\rm Im}(\varphi)$ is not cotorsion, and hence,  $H_2(\widehat{F})$ is not cotorsion. 
 
  \section{Lie algebras}
 
 In this paper by a Lie algebra we always mean a Lie algebra over $\mathbb{Z}.$ In this section we discuss versions of theorems A and B for the case of Lie algebras. We are not going to give detailed proofs because they are very similar to the case of groups. We just write down some preliminary results for Lie algebras which form a basis for the similar proofs.

 There are several non-equivalent definitions of homology of a Lie algebra over $\ZZ$ 
 (see \cite{ivanov2020homology}) but all of them coincide if we are interested only in the second homology \cite[Th.8.4]{ivanov2020homology}. For example, we can define $H_n(\gg)={\sf Tor}_n^{U\gg}(\ZZ,\ZZ),$ where $U\gg$ is the universal enveloping algebra. Then for any presentation of a Lie algebra as a quotient of a free Lie algebra $\gg=\ff/\rr,$ we have an analogue of Hopf's isomorphism 
 \begin{equation}
    H_2(\gg)\cong \frac{\rr\cap [\ff,\ff]}{[\rr,\ff]},
 \end{equation}
 which is natural in the short exact sequence $0\to \rr\to \ff\to \gg\to 0.$ 
 Moreover, this can be generalized to the following proposition.

 \begin{proposition}[cf. Prop.  \ref{prop:hopf}]
 \label{prop:Hopf_formula_for_Lie_alg}
 For any ideal $\mathfrak{u}$ of a Lie algebra $\mathfrak{g}$ there is an exact sequence
\begin{equation}
    H_2(\gg) \to H_2(\gg/\mathfrak{u}) \to \frac{\mathfrak{u}\cap [\gg,\gg]}{[\mathfrak{u},\gg]} \to 0,
\end{equation}
which is natural in the short exact sequence $0\to \uu \to \gg\to \gg/\uu\to 0.$
\end{proposition}
\begin{proof}
Consider a presentation $\gg=\ff/\rr.$ Denote by $\mathfrak{s}$ the preimage of $\uu$ in $\ff.$ Then $\gg/\uu \cong \ff/\mathfrak{s}.$ Therefore the cokernel of the homomorphism $H_2(\gg)\to H_2(\gg/\uu)$ is isomorphisc to the cokernel of $ \frac{\rr \cap [\ff,\ff] }{[\rr,\ff]}  \to \frac{\mathfrak{s} \cap [\ff,\ff]}{[\mathfrak{s},\ff]},$ which is equal to $A:=\frac{\mathfrak{s} \cap [\ff,\ff]}{[\mathfrak{s},\ff]+\rr\cap[\ff,\ff]}.$ By the modular law we have 
\[
[\mathfrak{s},\ff]+\rr\cap[\ff,\ff] = ([\mathfrak{s},\ff]+\rr)\cap[\ff,\ff]=([\mathfrak{s},\ff]+\rr)\cap \mathfrak{s}\cap[\ff,\ff].
\] By the second and the third  isomorphism theorems we have 
\[A= \frac{ \mathfrak{s}\cap [\ff,\ff] }{([\mathfrak{s},\ff]+\rr)\cap (\mathfrak{s} \cap [\ff,\ff])} = \frac{\mathfrak{s}\cap [\ff,\ff]+\rr}{[\mathfrak{s},\ff] + \rr}= \frac{\uu\cap [\gg,\gg]}{[\uu,\gg]} .\]
\end{proof} 

In the proof of Theorem B we used the fact that for a finitely generated free group $F$ there is an isomorphism $F_{ab}\cong \widehat{F}_{ab}$ that was proved in  \cite[Th. 13.3 (iv)]{bousfield1977homological} and \cite[Th.2.1]{baumslag1977inverse}. Here we prove an analogue of this result for Lie algebras.

\begin{lemma}[{cf. \cite{baur1980note}}]
\label{lemma:central_quotients_iso_Lie_alg}
Let $\gg$ be a finitely generated Lie algebra. Then the map $\eta:\gg\to \widehat{\gg}$ induces an isomorphism
\begin{equation}
\gg/\gamma_n(\gg)\cong \widehat{\gg}/\gamma_n( \widehat{\gg}). 
\end{equation}
In particular, $\gg_{ab}\cong \widehat{\gg}_{ab}.$
\end{lemma}
\begin{proof} By definition of $\widehat{\gg}$ we have a map $\widehat{\gg}\epi \gg/\gamma_n(\gg).$ We just need to prove that $\gamma_n(\widehat{\gg})={\sf Ker}( \widehat{\gg}\epi \gg/\gamma_n(\gg)).$ Since $\gg/\gamma_n(\gg)$ is nilpotent of class $\leq n-1,$ we obtain $\gamma_n(\widehat{\gg})\subseteq {\sf Ker}( \widehat{\gg}\epi \gg/\gamma_n(\gg)).$ So we only need to prove ${\sf Ker}( \widehat{\gg}\epi \gg/\gamma_n(\gg))   \subseteq \gamma_n(\widehat{\gg}) .$ The proof is by induction on $n.$ For $n=1$ the statement is obvious. Prove the step.

It is easy to see that any element $a\in {\sf Ker}( \widehat{\gg}\epi \gg/\gamma_n(\gg))$ can be written as $a=\sum_{i=n}^\infty \eta(a_i),$ where $a_i\in \gamma_i(\gg).$ If we denote by  $x_1,\dots,x_n$ a set of generators of $\gg$,
then any element $b\in \gamma_i(\gg)$ for $i>1$ can be  presented as $b=[b^{(1)},x_1]+\dots +[b^{(n)},x_n],$ where $b^{(1)},\dots,b^{(n)}\in \gamma_{i-1}(\gg).$   Therefore $$a=\sum_{i=n}^\infty \eta(a_i)=[ b^{(1)},\eta(x_1)]+\dots + [b^{(n)},\eta(x_n)],$$ where $b^{(j)}=\sum_{i=n}^\infty \eta(a_i^{(j)})$ and $a_i^{(j)}\in \gamma_{i-1}(\gg).$ Hence $b^{(j)}\in {\rm Ker}(\widehat{\gg}\epi \gg/\gamma_{n-1}(\gg)).$ By induction hypothesis we have $b^{(j)}\in \gamma_{n-1}(\widehat{\gg}).$ Therefore, $a\in \gamma_n(\widehat{\gg}).$
\end{proof}

\vspace{0.3cm}\noindent {\bf Theorem B for Lie algebras}. {\it Let $\gg$ be a finitely generated Lie algebra (over $\ZZ$) and $\ff\epi \gg$ be a presentation, where $\ff$ is a free Lie algebra of finite rank. Then the cokernel of the map
\begin{equation}
H_2(\widehat{\ff})\oplus H_2(\gg) \longrightarrow H_2(\widehat{\gg})
\end{equation}
is cotorsion. }
\vspace{0.3cm}
\begin{proof}
The proof is similar to the proof of Theorem B for groups (see Section~\ref{subsection:proof1}), so we provide only a short sketch here. Set $\rr$ to be the kernel of $\ff\to \gg$ and $\mathfrak{h}^n = \gamma_n(\mathfrak{h})$ for any $\mathfrak{h}$. Arguing as in the proof of Theorem~B for groups (with Proposition~\ref{prop:Hopf_formula_for_Lie_alg} used instead of the Hopf's formula) one can see that it is enough to prove that the cokernel of the map
\[
    \frac{\rr\cap \ff^2}{[\rr,\ff]}\longrightarrow \frac{\overline{\rr}\cap \widehat{\ff}^2}{[\overline{\rr}, \widehat{\ff}]}
\]
is cotorsion, where $\overline{\rr} = \Ker(\widehat{\ff}\to\widehat{\gg})\cong \varprojlim \rr/(\rr\cap \ff^n)$.

Due to Lemma~\ref{lemma:central_quotients_iso_Lie_alg} we have $\ff_{ab}\cong\widehat{\ff}_{ab}$, hence the same arguments as in the proof of Theorem~B for groups imply that 
\[
    \overline{\rr}\cap \widehat{\ff}^2\cong \varprojlim(\rr\cap \ff^2)/(\rr\cap \ff^n).
\]
Since the image of the subgroup $\rr\cap \ff^2$ in $\overline{\rr}\cap \widehat{\ff}^2$ is dense with respect to the inverse limit topology, we conclude using Theorem~\ref{th:barnea_shelah}.
\end{proof}

A module $M$ over a Lie algebra $\gg$ is a module over the universal enveloping algebra $U\gg.$ In other words, a module $M$ is an abelian group together with a homomorphism of Lie algebras $\gg \to {\sf End}_{\ZZ}(M).$ If $\eta: \gg\to U\gg$ is the embedding, then we set $[a,m]=\eta(a)m$ for any $a\in \gg$ and $m\in M.$ We also set $M^\gg=\{ m\in M\mid [m,a]=0 \text{ for } a\in \gg \}$ and $M_\gg=M/[\gg,M].$ For any two $\gg$-modules $M$ and $N$ we define a structure of $\gg$-module on $M\otimes N$ by the formula $[a,m\otimes n] = [a,m] \otimes n + m\otimes [a,n]$ for $a\in \gg.$ Note that the abelian subgroup of $M\otimes M$ generated by the elements of the form $m\otimes m$ for $m\in M$ is a $\gg$-submodule because of the equation
$$a(m\otimes m)=(am+m)\otimes (am+m) - (am)\otimes (am) - m\otimes m,$$ and we can consider $\Lambda^2 M$ as a $\gg$-module. 
For any $\gg$-module $M$ we consider the semidirect product $\gg \ltimes M,$ whose underlying abelian group is $\gg\oplus M$ and the bracket is defied by $[(a_1,m_1),(a_2,m_2)]=([a_1,a_2],[a_1,m_2]-[a_2,m_1]).$

\begin{lemma}[{cf. Lemma  \ref{lemma:secondHomologyOfCltimesM}}]
Let $M$ be a $\ZZ[x]$-module such that $M^{\ZZ x}=0.$ Then there is a natural isomorphism 
\begin{equation}
H_2(\ZZ x\ltimes M)\cong (\Lambda^2 M)_{\ZZ x}.    
\end{equation}
\end{lemma}
\begin{proof}
By \cite[Th.8.4]{ivanov2020homology} we obtain that the second homology of a Lie algebra $\gg$ can be computed using the Chevalley-Eilenberg coalgebra ${\sf CE}_\bullet(\gg).$ In other words, $H_2(\gg)$ is naturally isomorphic to the homology of the complex
\begin{equation}
    \Lambda^3 \gg \overset{d_3}\longrightarrow \Lambda^2 \gg \overset{d_2}\longrightarrow \gg,
\end{equation}
where $d_2(a \wedge b) = [a,b]$ and 
\[
d_3(a\wedge b \wedge c)=[a,b]\wedge c + [b,c]\wedge a + [c,a] \wedge b.
\] Then we just need to make a computation using this complex.

Since for any two abelian groups $A,B$ we have $\Lambda^*(A\oplus B)\cong \Lambda^*(A)\otimes \Lambda^*(B),$
in our case $\gg=\ZZ x \ltimes M $ we obtain $\Lambda^2 \gg \cong (\ZZ x\otimes  M) \oplus (\Lambda^2(M)) $ and $\Lambda^3 \gg = (\ZZ x \otimes  \Lambda^2 M) \oplus   (\Lambda^3 M).$ The map $\Lambda^2(M)\to \gg$ is trivial by the definition of $\ZZ x \ltimes M$ and the map $\ZZ x\otimes M \to \gg $ is defined by the formula $nx \otimes m\mapsto (0,nxm).$ Since $M^{\ZZ x}=0,$ we have 
\[{\sf Ker}(\Lambda^2 \gg \to \gg )=\Lambda^2 M.\]
The map $\Lambda^3 M \to \Lambda^2 \gg$ is trivial and the map $\ZZ x \otimes \Lambda^2 M \to \Lambda^2 M$ is given by $d_3(nx \wedge m_1 \wedge m_2 ) = nx m_1 \wedge m_2 - nx m_2 \wedge m_1 = nx(m_1\wedge m_2).$ Hence  
\[{\sf Im}( \Lambda^3 \gg \longrightarrow \Lambda^2 \gg ) = x(\Lambda^2 M).\]
Therefore $H_2(\gg)=(\Lambda^2 M)_{\ZZ x}.$
\end{proof}

We define the lamplighter Lie algebra as the semidirect product 
\begin{equation}
\ll=\ZZ x \ltimes \ZZ[x].
\end{equation}
Then it is easy to check that $\widehat{\ll} \cong  \ZZ x \ltimes \ZZ\llb x \rrb.$

\begin{theorem}[{cf. Th. \ref{th:double_lamplighter}}]
The cokernel of the map 
\begin{equation}
H_2(\ll) \longrightarrow H_2( \widehat{\ll} )
\end{equation}
is not cotorsion. 
\end{theorem}
\begin{proof}
The proof is very similar to the proof of Theorem \ref{th:double_lamplighter}. In this proof we have to use Theorem \ref{th:power_series_cotorsion} with $P(x,y)=x+y$ (while in the case of groups we used $P(x,y)=x+y+xy.$) 
\end{proof}

\vspace{0.3cm}\noindent{\bf Theorem A for Lie algebras.} {\it Let $\ff$ be a free Lie algebra (over $\ZZ$) with at least two generators. Then $H_2(\widehat{\ff})$ is not cotorsion.} 
\vspace{0.3cm}
\begin{proof}
The proof repeats literally the proof of Theorem A for groups (see Section~\ref{sec:proof_of_TheoremA}).
\end{proof}

 \section{Proofs of Remarks}
\noindent {\it Proof of Remark \ref{remark1}}
First pick a free group $F$ and a homomorphism $F\to G$ which induces an isomorphism $H_1(F)\to H_1(G)$. The statement then follows by induction on $n$. Suppose that the above homomorphism induces an isomorphism $F/\gamma_n(F)\to G/\gamma_n(G)$. There is a natural exact sequence
$$
H_2(G)\to H_2(G/\gamma_n(G))\to \gamma_n(G)/\gamma_{n+1}(G)\to 1.
$$
The left hand arrow is zero, since $H_2(G/\gamma_n(G))=H_2(F/\gamma_n(F))=\gamma_n(F)/\gamma_{n+1}(F)$ is free abelian and ${\sf Hom}(H_2(G),\mathbb Z)=0$. Hence the map $F\to G$ induces an isomorphism $\gamma_n(F)/\gamma_{n+1}(F)\to \gamma_n(G)/\gamma_{n+1}(G)$ and we can make the needed inductive step.

\vspace{.5cm}\noindent{\it Proof of Remark \ref{remark2}}
For a connected space $X$, $H_1(X)=H_1(\pi_1(X))$. Recall that (see \cite{ivanov2021right}), there is a natural short exact sequence 
$$
1\to {\varprojlim}^1 M_n(\pi_1(X))\to \pi_1({\mathbb Z}_\infty X)\to \widehat{\pi_1(X)}\to 1,
$$
where $M_n$ are so-called Baer invariants. Here we don't need a definition of the Baer invariants, we will use only the fact that $M_n$-s are certain abelian groups, for details see \cite{ivanov2021right}. Denote $H:={\sf coker}\{H_1(X)\to H_1({\mathbb Z}_\infty X)\}.$ The above description of the group $\pi_1({\mathbb Z}_\infty X)$ imply that there is a natural short exact sequence of abelian groups
\begin{equation}\label{lim1lim}
0\to \overline{{\varprojlim}^1 M_n(\pi_1(X))}\to H\to coker\{H_1(\pi_1(X))\to H_1(\widehat{\pi_1(X)})\}\to 0. 
\end{equation}
Here $\overline{{\varprojlim}^1 M_n(\pi_1(X))}$ is a certain quotient of ${\varprojlim}^1 M_n(\pi_1(X))$. Any quotient of a ${\varprojlim}^1$ of abelian groups is a cotorsion group (see Section~\ref{subsec:cotorsion_general} for the properties of cotorsion groups). That is, the group $\overline{{\varprojlim}^1 M_n(\pi_1(X))}$ is cotorsion. The right hand side of (\ref{lim1lim}) also is cotorsion by the result of Barnea and Shelah \cite{barnea2018abelianization}. The class of cotorsion groups is closed under extensions and the needed statement follows.

\vspace{.5cm}\noindent{\it Proof of Remark \ref{remark3}.} Let $\ff$ denote the Lie algebra freely generated by $a_1,\dots,a_n.$ Then $H_2(\widehat{\ff})$ is the homology of the complex 
\begin{equation}
\Lambda^3 (\widehat{\ff}) \overset{d_3}\longrightarrow \Lambda^2 (\widehat{\ff}) \overset{d_2}\longrightarrow \widehat{\ff},
\end{equation}
where $d_2(a\wedge b)=[a,b]$
and
\begin{equation}
d_3(a\wedge b \wedge c ) = [a,b] \wedge c + [b,c]\wedge a + [c,a] \wedge b   
\end{equation}
(see \cite[Th.8.4]{ivanov2020homology}). We need to prove that the subgroup of $H_2( \widehat{\ff} )$ consisting of elements that can be presented as $ \alpha_1 \wedge a_1 + \dots + \alpha_n \wedge a_n$ for some $\alpha_1,\dots,\alpha_n\in \widehat{\ff}$ is cotorsion. 

For any Lie algebra $\gg$ with some fixed elements $a_1,\dots,a_n\in \gg$ we consider the maps 
\begin{equation}
\theta_{\gg} : \gg^n \longrightarrow \Lambda^2 \gg, \hspace{1cm} \theta_{\gg}(\alpha_1,\dots,\alpha_n)=\alpha_1 \wedge a_1 +\dots + \alpha_n \wedge a_n,    
\end{equation}
\begin{equation}
\tau_\gg : \gg^n \longrightarrow \gg, \hspace{1cm} \tau_\gg(\alpha_1,\dots,\alpha_n)=[\alpha_1,a_1] + \dots + [\alpha_n,a_n].    
\end{equation}
Note that $\tau_\gg=d_2 \theta_\gg.$
We also set $K(\gg) := {\sf Ker}(\tau_\gg).$ Then we obtain a map 
\begin{equation}
\theta'_\gg: K(\gg) \longrightarrow H_2(\gg)
\end{equation}
induced by $\theta_\gg.$

In this terms we need to prove that the image of $\theta'_{\widehat\ff} : K(\widehat{\ff})
\longrightarrow
H_2(\widehat \ff)$ is cotorsion. 
Since we have a commutative diagram 
\begin{equation}
\begin{tikzcd}
K(\ff)\arrow[r,"\theta'_{\ff}"] \arrow[d] & H_2( \ff ) \arrow[d] \\  K(\widehat{\ff})\arrow[r,"\theta'_{\widehat{\ff}}"] & H_2(\widehat{\ff}),  
\end{tikzcd}
\end{equation}
and $H_2(\ff)=0,$ we obtain $\theta'_{\widehat{\ff}}( K(\ff) )=0$ (here we identify $\ff$ with its image in $\widehat{\ff}$ and assume $K(\ff)\subseteq K(\widehat{\ff})$).   Therefore it is enough to prove that $K(\widehat{\ff})/K(\ff)$ is cotorsion because  the image of $\theta'_{\widehat{\ff}}:K(\widehat{\ff})\to H_2(\widehat{\ff})$ is its quotient. 

Let's prove it. The algebra $\ff$ has a natural grading $\ff=\bigoplus_{i\geq 1} \ff_i$ such that $\gamma_k(\ff)=\bigoplus_{i\geq k} \ff_i.$ It follows that $\widehat{\ff}=\prod_{i\geq 1} \ff_i.$  
Therefore, if we set $K_i:={\rm Ker}( \theta_i : \ff_i^n \to \ff_{i+1}),$ where $\theta_i$ is the restriction of $\theta_\ff,$ then $K(\ff)=\bigoplus_{i\geq 1} K_i$ 
and $K( \widehat{\ff})=\prod_{i\geq 1} K_i.$ It follows that, $K(\widehat{\ff})/K(\ff)=(\prod_{i\geq 1} K_i) / (\bigoplus K_i)$ is cotorsion by the theorem of Hulanicki.

\printbibliography

\end{document}